\documentclass[12pt]{amsart}

 \usepackage[utf8]{inputenc}

 \usepackage[english]{babel}

 \usepackage{amsmath,amsfonts,amssymb}
 \usepackage{mathrsfs}
 \usepackage{dsfont} 
 \usepackage{mathtools} 
 \usepackage{esint} 
 \usepackage{xfrac} 
 \usepackage{defs} 
 \usepackage{fonts} 

 \usepackage{amsthm}
 \usepackage[shortlabels]{enumitem} 
 \usepackage[colorlinks=true,urlcolor=blue, citecolor=red,linkcolor=blue,linktocpage,pdfpagelabels, bookmarksnumbered,bookmarksopen]{hyperref} 
 
 \usepackage{cleveref}

 \usepackage[paper=a4paper,hmargin=1.2in,vmargin=1in, marginparwidth=1in]{geometry} 

 \usepackage[colorinlistoftodos,prependcaption,linecolor=blue,backgroundcolor=blue!25,bordercolor=blue,textsize=tiny]{todonotes}
 \usepackage{cite} 

plus 6pt minus 12pt
\numberwithin{equation}{section}

\newtheorem{theorem}{Theorem}[section]
\newtheorem{proposition}[theorem]{Proposition}
\theoremstyle{definition}

\theoremstyle{plain}
\newtheorem{lemma}[theorem]{Lemma}
\newtheorem{corollary}[theorem]{Corollary}
\theoremstyle{remark}
\newtheorem{remark}[theorem]{Remark}

\title[Weak Limits of Fractional Sobolev Homeomorphisms]{Weak Limits of Fractional Sobolev Homeomorphisms are Almost Injective: A Note}
 \author{Armin Schikorra}
\address[Armin Schikorra]{Department of Mathematics,
University of Pittsburgh,
301 Thackeray Hall,
Pittsburgh, PA 15260, USA}
\email{armin@pitt.edu}%
 \author{James M. Scott}
\address[James M. Scott]{Department of Mathematics,
University of Pittsburgh,
301 Thackeray Hall,
Pittsburgh, PA 15260, USA}
\email{james.scott@pitt.edu}%

\newcommand{\fstar}{f^{\ast}}

\newcommand{\osc}{{\rm osc}}
\newcommand{\R}{\mathbb{R}}
\newcommand{\eps}{\varepsilon}
\renewcommand{\S}{\mathbb{S}}
\newcommand{\Z}{\mathbb{Z}}
\newcommand{\N}{\mathbb{N}}
\newcommand{\brac}[1]{\left ( #1 \right )}
\newcommand{\abs}[1]{\left | #1 \right |}

\begin{document}

\begin{abstract}
Let $\Omega \subset \R^n$ be an open set and $f_k \in W^{s,p}(\Omega;\R^n)$ be a sequence of homeomorphisms weakly converging to $f \in W^{s,p}(\Omega;\R^n)$. It is known that if $s=1$ and $p > n-1$ then $f$ is injective almost everywhere in the domain and the target. In this note we extend such results to the case $s\in(0,1)$ and $sp > n-1$. This in particular applies to $C^s$-H\"older maps. 
\end{abstract}

\maketitle

\section{Introduction and main result}
The goal of this note is to prove the following theorem:

\begin{theorem}\label{thm:MainThm}
Let $\Omega \subset \bbR^n$, $n \geq 2$, be open and let $f : \Omega \to \bbR^n$ be a weak $W^{s,p}$-limit of Sobolev homeomorphisms $f_j \in W^{s,p}(\Omega;\bbR^n)$ with $sp > n-1$. Then there is a representative $\wh{f}$ and a set $\Gamma \subset \bbR^n$ of Hausdorff dimension $\frac{n-1}{s}$ such that $(\wh{f})^{-1}(y)$ consists of only one point for every $y \in \wh{f}(\Omega) \setminus \Gamma$.
\end{theorem}
For definitions we refer to the next section. An immediate corollary of \Cref{thm:MainThm} and the embedding $C^{s} \hookrightarrow W^{s-\eps,p}_{loc}$ for any $\eps > 0$ is the following statement for H\"older maps.
\begin{corollary}\label{co:hoelderversion}
Let $\Omega \subset \bbR^n$, $n \geq 2$ be open and let $f \in C^s(\Omega;\R^n)$ be the pointwise limit of a sequence of equibounded homeomorphisms $f_j \in C^{s}(\Omega;\bbR^n)$.
If $s>\frac{n}{n-1}$, then there is a set $\Gamma \subset \bbR^n$ of Hausdorff dimension $\frac{n-1}{s}$ such that $(f)^{-1}(y)$ consists of only one point for every $y \in f(\Omega) \setminus \Gamma$.
\end{corollary}
Observe that for $s \leq \frac{n-1}{n}$ the above statements hold trivially.

This note is inspired by the recent work by Bouchala, Hencl, and Molchanova \cite{BHM19} who proved a corresponding result for $s=1$.
\begin{theorem}[Bouchala, Hencl, Molchanova]\label{th:bhm:1}
Let $f : \Omega \to \bbR^n$ be a weak limit of Sobolev homeomorphisms $f_j \in W^{1,p}(\Omega;\bbR^n)$ with $p > n-1$. Then 
there is a representative $\wh{f}$ and a set $\Gamma \subset \bbR^n$ of Hausdorff dimension $n-1$ such that $(\wh{f})^{-1}(y)$ consists of only one point for every $y \in \wh{f}(\Omega) \setminus \Gamma$.
\end{theorem}
While \Cref{th:bhm:1} (and in turn our \Cref{thm:MainThm}) follows an adaptation of the arguments in the seminal work by M\"uller and Spector \cite{MS95}, Bouchala, Hencl, and Molchanova \cite{BHM19} also provide an example of the limit case $p=n-1$, where a theorem such as \Cref{th:bhm:1} completely fails. Namely they showed
\begin{theorem}[Bouchala, Hencl, Molchanova]
For $n\geq 3$ there exists $f: [-1,1]^{n} \to [-1,1]^n$ and a strong limit of Sobolev homeomorphisms $f_k \in W^{1,n-1}([-1,1]^n,\R^n)$ with $f_k(x) = x$ on the boundary $\partial [-1,1]^n$ and such that there exists a set $\Gamma \subset [-1,1]^n$ of positive Lebesgue measure and $f^{-1}(y)$ is a nontrivial continuum for every $y \in \Gamma$.
\end{theorem}

As the authors of \cite{BHM19} mention, it may seem surprising that the Hausdorff dimension of the critical set $\Gamma$ seems to suddenly jump from $n-1$ to $n$ as $p$ changes from $p>n-1$ to $p=n-1$. This question served as one motivation to study the situation for fractional Sobolev spaces.

Let us stress that \Cref{thm:MainThm} follows a very similar argument as the $s=1$ proof of \Cref{th:bhm:1} in \cite{BHM19}, which in turn is a streamlined argument of 
known results and techniques from earlier works, see \cite{BHM17,MS95,MST96}. Indeed, a crucial fact that is used for $s=1$ is that on ``good slices'' $\partial B_r$ the $f_k$ converge in $W^{1,p}(\partial B_r)$, and so  using Sobolev-Morrey embedding on these $n-1$-dimensional slices the $f_k$ in fact converge uniformly if $p > n-1$. If $p=n-1$ this uniform convergence may fail.

The same is true if the $f_k$ converge in $W^{s,p} (\partial B_r)$ for  good slices $\p B_r$ and $s \in (0,1)$: if $sp > n-1$ then the convergence is uniform on $\p B_r$, and if $sp= n-1$ it may not. 

But somewhat surprisingly, a result such as \Cref{thm:MainThm} and in particular \Cref{co:hoelderversion} seems to be unknown to some experts, and the authors thought it important to be available in the literature.

We try to keep this note as self-contained as possible. In \Cref{s:sobolev} we gather the main results on Sobolev spaces that we work with. In \Cref{s:degreemon} we discuss the needed notions of degree, and show monotonicity of the degree for limits of homeomorphisms.  In \Cref{s:coforlemma31}, we collect the corollaries for the topological image from the previous section. In \Cref{s:thmainproof} we prove our main theorem.

{\bf Acknowledgment:}
the authors thank Anastasia Molchanova and Daniel Campbell for informative and illuminating discussions.

AS is supported by Simons foundation, grant no 579261.

\section{Preliminaries on Sobolev spaces, capacities etc.}\label{s:sobolev}

In this section we establish notation. For $s \in (0,1)$ and $p \in (1,\infty)$ we denote the classes of functions $u : \Omega \to \bbR^n$ for which the Gagliardo seminorm 
\begin{equation}
	[u]_{W^{s,p}(\Omega)}^p := \iintdm{\Omega}{\Omega}{\frac{|f(x)-f(y)|^p}{|x-y|^{n+sp}}}{y}{x}
\end{equation}
is finite as the fractional Sobolev spaces $W^{s,p}(\Omega;\bbR^n)$, with norm $\Vnorm{u}_{W^{s,p}(\Omega)}^p := \Vnorm{u}_{L^p(\Omega)}^p + [u]_{W^{s,p}(\Omega)}^p$.

We denote the $n$-dimensional Lebesgue measure of a set $A \subset \bbR^n$ by $\cL^n(A)$, and for $\beta > 0$ we denote the $\beta$-dimensional Hausdorff measure by $\scH^{\beta}(A)$. We use the convention $A \precsim B$ whenever there exists a constant $C$ such that $A \leq CB$.

Define the \textit{precise representative} of a measurable function $f$ by
\begin{equation}\label{def:PreciseRep}
\fstar(x) := 
\begin{cases}
	\lim\limits_{r \to 0^+} \fint_{B_r(x)} f(y) \, \rmd y \,, & \quad \text{ when the limit exists, } \\
	0\,, & \quad \text{ otherwise.}
\end{cases}
\end{equation}

Many properties of the precise representative for functions in the \textit{Bessel potential} spaces are accessible in the literature. The corresponding statements can then be obtained for fractional Sobolev functions via embedding theorems for the Triebel-Lizorkin spaces $F^s_{p,q}$; see \cite{RS11}. For completeness, we gather here a summary of the statements we will need.

We denote the Bessel potential spaces $H^{s,p}$ by
\begin{equation}
H^{s,p}(\bbR^n;\bbR^n) := \Vbrack{ f : \bbR^n \to \bbR^n \, : \, \vnorm{ \cF^{-1} \vparen{ \vparen{1+|\xi|^2}^{s/2}  (\cF f) (\xi) } }_{L^p(\bbR^n)}  < \infty  }\,,
\end{equation}
where $\cF$ and $\cF^{-1}$ denote the Fourier transform and its inverse respectively.
The following is a corollary of a classical embedding theorem for the spaces $F^s_{p,q}$ \cite[Section 2.2.3]{RS11}, \cite[Theorem 2.14, Remark 2.4]{S18}. We are additionally using the identifications $F^s_{p,2} = H^{s,p}$ and $F^s_{p,p} = W^{s,p}$.
\begin{theorem}\label{thm:TriebelEmbedding}
Let $N \geq 1$.
Let $p \in (1,\infty)$ and $s \in (0,1)$, and suppose that $t \in (0,1)$ and $p_t \in (1,\infty)$ satisfy 
\begin{equation*}
s - \frac{N}{p}  < t < s\,, \qquad p_t := \frac{Np}{N-(s-t)p}\,.
\end{equation*}
Then
\begin{equation}
W^{s,p}(\bbR^N) \hookrightarrow H^{t,p_t}(\bbR^N)\,, \quad \text{ or } \quad [f]_{H^{t,p_t}(\bbR^N)} \precsim [f]_{W^{s,p}(\bbR^N)}\,.
\end{equation}
\end{theorem}

Note that if we write the definition of $p_t$ as
\begin{equation}\label{eq:ExpRelation1}
\begin{split}
sp-N = \frac{p}{p_t}(tp_t-N)\,, \\
\end{split}
\end{equation}
then it becomes clear that if $sp > N$ then $tp_t > N$ for any $t \in (0,s)$.

With this embedding we can prove some useful properties of the precise representative:

\begin{proposition}\label{prop:HMeasureOfNonLebPts}
Suppose $f \in W^{s,p}(\bbR^n;\bbR^n)$ with $sp \in [1,n)$. Let $p^* = \frac{np}{n-sp}$. Define
\begin{equation}\label{eq:LebPtSet}
A_{sp} := \Vbrack{x \in \bbR^n \, : \, x \text{ is not a Lebesgue point of } f }\,.
\end{equation}
Then the following hold: 
\begin{enumerate}[(i)]
\item\label{item:LebPts:Measure} $\dim_{\scH}(A_{sp}) \leq n-sp$.

\item\label{item:LebPts:Embedding} For any $x \in \bbR^n \setminus A_{sp}$,
\begin{equation}\label{eq:LebPoints}
\lim\limits_{r \to 0^+} \fint_{B_r(x)}|f(y) - \fstar(x)|^{q} \, \rmd y = 0\,,
\end{equation}
for every $q \in [1,p^*)$.

\item\label{item:LebPts:Mollifiers} If $\varphi_{\veps}$ is the family of standard mollifiers then
\begin{equation*}
\varphi_{\veps} \ast f(x) \to \fstar(x)
\end{equation*}
for each $x \in \Omega \setminus A_{sp}$.
\end{enumerate}
\end{proposition}

\begin{proof}
Let $\veps > 0$ be arbitrary; we will show that 
\begin{equation}\label{eq:ProofofHausdorffMeasure1}
\scH^{n-sp+\veps}(A_{sp}) =0\,,
\end{equation}
which will imply \ref{item:LebPts:Measure}. 
We use Theorem \ref{thm:TriebelEmbedding} with $N = n$; choose $t \in (0,s)$ so that
\begin{equation*}
n-tp_t = n-sp+\veps\,;
\end{equation*}
this is possible since by definition $n-tp_t > n -sp$  for $sp \in [1,n)$ and for any $t \in (0,s)$. Then $f \in H^{t,p_t}(\bbR^n;\bbR^n)$ and so \cite[Proposition 6.1.2, Theorem 5.1.13]{A12} implies 
$\scH^{\beta}(A_{sp})=0$ for all $\beta \geq n-tp_t = n-sp+\veps$, and so \eqref{eq:ProofofHausdorffMeasure1} is established. 

To see \ref{item:LebPts:Embedding}, use Theorem \ref{thm:TriebelEmbedding} with $N = n$ again; note that any $q \in (p,p_*)$ can be written $q = p_t$ for some $t \in (0,s)$. Then $f \in H^{t,p_t}(\bbR^n;\bbR^n)$ for every $t \in (0,s)$, and so \cite[Theorem 6.2.1]{A12} applies, which is precisely \ref{item:LebPts:Embedding}. We obtain \eqref{eq:LebPoints} for the range $q \in [1,p]$ using H\"older's inequality.

For a proof of \ref{item:LebPts:Mollifiers} see \cite[Theorem 4.1, (iv)]{E15}.
\end{proof}

\begin{lemma}\label{lma:WeaktoStrongINV}
Suppose $f \in W^{s,p}(\bbR^n ; \bbR^n)$ with $sp \in [1,n)$, and suppose $\fstar(x) \in E$ for every $x \in \bbR^n \setminus M$, where $\cL^n(M) = 0$ and $E \subset \bbR^n$ is a closed set. Then $\fstar(x) \in E$ for every $x \in \bbR^n \setminus A_{sp}$. 
\end{lemma}

\begin{proof}
Suppose to the contrary, that $\fstar(x) \in \bbR^n \setminus E$ for some $x \in \bbR^n \setminus A_{sp}$. Then there exists $\veps > 0$ such that $B(\fstar(x),\veps) \subset \bbR^n \setminus E$. By assumption that $\fstar(y) \in E$ for $y \in \bbR^n \setminus M$
\begin{equation*}
\fint_{B(x,r)} |f(y) - \fstar(x)|^p \, \rmd y = \fint_{B(x,r) \setminus M} |f(y) - \fstar(x)|^p \, \rmd y \geq \veps^p 
\end{equation*}
uniformly as $r \to 0$, which is a contradiction since $\fstar$ satisfies \eqref{eq:LebPoints} for every $x \in \bbR^n \setminus A_{sp}$.
\end{proof}

We will need information on the Hausdorff dimension of images of spheres embedded in $\bbR^n$.
The following is a special case of such a result in \cite{HH15} for Bessel potential functions, which will then apply to functions in $W^{s,p}$ via Theorem \ref{thm:TriebelEmbedding}:

\begin{proposition}[\cite{HH15}, Theorem 1.1]\label{prop:HenclDistortionThm}
Let $N$, $K \in \bbN$, $t \in (0,1)$ and $q \in (1,\infty)$ with $tq > N$ and $\alpha \in (0,N]$. Define $\beta := \frac{\alpha q}{tq - N + \alpha}$. Suppose $g \in H^{t,q}(\bbR^N;\bbR^K)$ is a continuous representative and $A \subset \bbR^N$ is a set with $\dim_{\scH}(A) \leq \alpha$. Then $\dim_{\scH}(g(A)) \leq \beta$.
\end{proposition}

We then have as a corollary

\begin{theorem}\label{thm:DimOfImage}
Let $n \geq 2$, $s \in (0,1)$, and $p > 1$ with $n-1 < sp < n$. Let $r>0$, $a \in \bbR^n$ with $\p B \equiv \p B(a,r)$ and $g \in W^{s,p}(\p B;\bbR^{n})$ be a continuous representative. Then $\dim_{\scH}(g(\p B)) \leq \frac{n-1}{s}$. 
\end{theorem}

\begin{proof}
It suffices to show that
\begin{equation*}
\scH^{\frac{n-1}{s}+\veps}(g(\p B)) = 0
\end{equation*}
for arbitrary $\veps > 0$ small. Cover $\p B$ by sets $S_i$ diffeomorphic to $\bbR^{n-1}$ ($2^n$ hemispheres will do), and let $\psi_i : \bbR^{n-1} \to S_i$ be the corresponding diffeomorphisms. So $\p B \subset \bigcup_{i=1}^M S_i $, and the functions
\begin{equation*}
g_i := g \circ \psi_i
\end{equation*}
belong to $W^{s,p}(\bbR^{n-1};\bbR^n)$, and hence belong to $H^{t,p_t}(\bbR^{n-1};\bbR^n)$ by Theorem \ref{thm:TriebelEmbedding} for any $t \in (s-\frac{n-1}{p},s)$ and for $p_t = \frac{(n-1)p}{(n-1)-(s-t)p}$.

Applying Proposition \ref{prop:HenclDistortionThm} to each $g_i$ with $q = p_t$ and $N = \alpha = n-1$ gives
\begin{equation*}
\scH^{\gamma}(g_i(\bbR^{n-1})) = 0\,, \quad \text{ for every } \gamma > \frac{n-1}{t}\,, \, i = \vbrack{1, \ldots, M}\,.
\end{equation*}
Choose $t < s$ close enough to $s$ so that $\frac{n-1}{t} < \frac{n-1}{s}+ \veps$. Then
\begin{equation*}
\scH^{\frac{n-1}{s}+\veps}(g(\p B)) \leq \sum_{i=1}^M \scH^{\frac{n-1}{s}+\veps}(g(S_i)) = \sum_{i=1}^M \scH^{\frac{n-1}{s}+\veps}(g_i(\bbR^{n-1})) = 0\,,
\end{equation*}
as desired.
\end{proof}

Throughout this note we additionally require control of fractional Sobolev functions on spheres in $\bbR^n$. In the local case, this control is obtained straightforwardly; for example, using Fubini's theorem for a smooth function $f$ on $\overline{B(a,r)}$ 
\begin{equation*}
\int_0^r \intdm{\p B(a,\rho)}{|\wt{\grad}f(\rho \omega)|^p }{\scH^{n-1}(w)} \, \rmd \rho \leq \intdm{B(a,r)}{|\grad f(x)|^p}{x}\,,
\end{equation*}
where $\wt{\grad}f$ denotes the tangential derivative of $f|_{\p B(a,\rho)}$. The following Besov-type inequality serves as a fractional analogue:

\begin{lemma}\label{lma:TangentDerivEstimate}
Let $B(a,r) \subset \bbR^n$, with $p \in [1,\infty)$ and $s \in (0,1)$. Then there exists a constant $C = C(n,s,p)$ such that for every $f \in W^{s,p}(B(a,r);\bbR^n)$
\begin{equation}\label{eq:TangentDerivEstimate}
\begin{split}
\int_{r/2}^r \iintdm{\p B(a,\rho)}{\p B(a,\rho)}{\frac{|f(x)-f(y)|^p}{|x-y|^{n-1+sp}}}{\scH^{n-1}(y)}{\scH^{n-1}(x)} \, \rmd \rho \leq C [f]_{W^{s,p}(B(a,r))}^p\,.
\end{split}
\end{equation}
\end{lemma}

These types of estimates are well-known to experts (see for example \cite{BBM05}), but for the sake of completeness we have included the proof in the appendix, see \Cref{s:lemma2}. The following corollary to the lemma reveals finer properties of Sobolev functions:

\begin{corollary}\label{cor:PrecRepOnSphere}
Let $1 < sp < n$, let $x_0 \in \Omega \subset \bbR^n$, and suppose $f \in W^{s,p}(\Omega;\bbR^n)$. Then there exists a set $N_{x_0} \subset (0,\dist(x_0,\p \Omega))$ with $\scL^1(N_{x_0}) = 0$ such that for every $r \in (0,\dist(x_0,\p \Omega)) \setminus N_{x_0}$ the function $\fstar|_{\p B(x_0,r)}$ belongs to $W^{s,p}(\p B(x_0,r);\bbR^n)$, where $\fstar$ is the precise representative defined in \eqref{def:PreciseRep}. If in addition $sp > n-1$ then $\fstar|_{\p B(x_0,r)}$ is continuous.
In general the singular set depends on $x_0$.
\end{corollary}

\begin{proof}
For $\veps > 0$ let $\varphi^{\veps}$ be the standard mollifier, and let $f^{\veps} := \varphi^{\veps} \ast \fstar$. Then $f^{\veps}$ converges to $\fstar$ in $W^{s,p}(B(x_0,r))$ for any $r \in (0,\dist(x_0,\p \Omega))$, and by \Cref{lma:TangentDerivEstimate}
\begin{equation*}
\int_{r/2}^{r} {[f^{\veps} - \fstar]_{W^{s,p}(\p B(x_0,\rho))}^p } \, \rmd \rho \leq C [f^{\veps} - \fstar]_{W^{s,p}(B(x_0,r))}^p \to 0 \text{ as } \veps \to 0\,.
\end{equation*}
Thus for $\scL^1$-almost every $r \in (0,\dist(x_0,\p \Omega))$ we have that the smooth functions $f^{\veps}|_{\p B(x_0,r)}$ converge to a function $g_r \in W^{s,p}(\p B(x_0,r))$.
On the other hand, Proposition \ref{prop:HMeasureOfNonLebPts} applies to $f$ since we can find a Sobolev extension domain $K$ satisfying $B(x_0,r) \subset K \subset \Omega$. Thus since $sp > 1$ we have from Proposition \ref{prop:HMeasureOfNonLebPts}\ref{item:LebPts:Mollifiers} that for every $r \in (0,\dist(x_0,\p \Omega))$
\begin{equation*}
f^{\veps}(x) \to \fstar(x) \text{ on } B(x_0,r) \setminus A_{sp}\,, \text{ where } \scH^{n-1}(A_{sp}) = 0\,.
\end{equation*}
Therefore for $\scL^1$-almost every $r \in (0,\dist(x_0,\p \Omega))$ the functions $f^{\veps}|_{\p B(x_0,r)}(x)$ converge to $\fstar(x)$ for $\scH^{n-1}$-almost every $x \in \p B(x_0,r)$. So for $\scL^1$-almost every $r \in (0,\dist(x_0,\p \Omega))$ the function $\fstar|_{\p B(x_0,r)}$ agrees with $g_r$ up to a set of $\scH^{n-1}$-measure zero, hence $\fstar|_{\p B(x_0,r)}$ belongs to $W^{s,p}(\p B(x_0,r))$.

Now if $sp > n-1$, then $f^{\veps} \to g_r$ locally uniformly on $\p B(x_0,r)$ by the Sobolev compact embedding theorem (see for example \cite[Theorem 2, pg. 82]{RS11},  \cite[Lemma 41.4]{T07}), and additionally
$
\scH^1(A_{sp}) = 0
$.
Therefore for $\scL^1$-almost every $r \in (0,\dist(x_0,\p \Omega))$ the sequence $f^{\veps}(x)$ converges to $\fstar(x)$
\textit{for every} $x \in \p B(x_0,r)$, and so $\fstar(x)$ agrees with the continuous function $g_r(x)$ \textit{for every} $x \in \p B(x_0,r)$.

\end{proof}

The following is an adaptation of \cite[Proposition 3.1]{LS19}, which in turn is an extension of an argument in \cite{S88}.
 
\begin{proposition}\label{pr:osciscontinuity}
Let $\Omega \subset \R^n$, $s \in (0,1)$ and $p \in (1,\infty)$ with $n-1 < sp < n$. Assume that $f \in W^{s,p}(\Omega;\R^n)$ satisfies the following: for any $x_0 \in \Omega$ there exists a set $N_{x_0}$ satisfying $\scL^1(N_{x_0})=0$ such that for all radii $r$, $\rho \in (0,\dist(x_0,\Omega)) \backslash N_{x_0}$ with $r < \rho$, there holds for some $\Lambda \geq 1$ independent of $r$, $\rho$ and $x_0$
\[
  \osc_{\partial B(x_0,r)} \fstar \leq \Lambda\, \osc_{\partial B(x_0,\rho)} \fstar\,,
\]
where $\fstar$ is the continuous representative of $f$ defined in \eqref{def:PreciseRep}.
Then there exists a singular set $\Sigma \subset \Omega$ with $\scH^{(n-sp)_+}(\Sigma) = 0$ 
such that $\fstar$ is continuous on $\Omega \backslash \Sigma$. 

\end{proposition}
\begin{proof}
Without loss of generality assume $sp < n$. The case $n=sp$ can be found in \cite[Proposition 3.1.]{LS19}, and $n < sp$ is obvious by Morrey-Sobolev embedding; see \cite{DNPV12}.

By Corollary \ref{cor:PrecRepOnSphere} for any $R > 0$ with $B(x_0,R) \subset \Omega$ and $\scL^1$-almost any $r < \rho < R$, the function $\fstar|_{\p B(x_0,\rho)}$ belongs to $W^{s,p}(B(x_0,\rho))$.
As in \cite[Proposition 3.1]{LS19}, by Morrey-Sobolev embedding
\[
 \brac{\osc_{\partial B(x_0,r)} \fstar}^p \leq \Lambda \brac{\osc_{\partial B(x_0,\rho)} \fstar}^p \leq C\rho^{sp-(n-1)} \int_{\partial B(x_0,\rho)}\int_{\partial B(x_0,\rho)} \frac{|\fstar(x)-\fstar(y)|^p}{|x-y|^{(n-1)+sp}}\, \rmd x\, \rmd y\,.
\]
Multiplying by $\rho^{-sp+(n-1)}$ and integrating in $\rho$ we obtain using \Cref{lma:TangentDerivEstimate}
\[
 c(s,p) \brac{R^{n-sp}-r^{n-sp}} \brac{\osc_{\partial B(x_0,r)} \fstar}^p \leq [\fstar]_{W^{s,p}(B(x_0,R))}^p\,.
\]
In particular we have
\begin{equation}\label{eq:osc:1}
 \brac{\osc_{\partial B(x_0,r)} \fstar}^p \leq R^{sp-n} [\fstar]_{W^{s,p}(B(x_0,R))}^p
\end{equation}
for any $R \in (0,\dist(x_0,\p \Omega))$ and for every $r \in( 0, R/2) \setminus N_{x_0}$.
Let 
\begin{equation}\label{eq:FrostLemmaSet}
 X := \left \{x \in \Omega: \limsup_{R \to 0^+} R^{sp-n} [\fstar]_{W^{s,p}(B(x,R))}^p > 0\right \}\,.
\end{equation}
By Frostman's Lemma (see \cite[Corollary 3.2.3]{Z89}) we have that $\scH^{(n-sp)_+}(X) = 0$.

Define $\Sigma = A_{sp} \cup X$, where $A_{sp}$ is defined in Proposition \ref{prop:HMeasureOfNonLebPts}. Let $x_0 \in \Omega \backslash \Sigma$ and let $\eps > 0$. Observe that if for some $R > 0$
\[
 R^{sp-n} [\fstar]_{W^{s,p}(B(x_0,R))}  < \eps
\]
then for all $y_0 \in B(x_0,R/2)$,
\[
 (R/2)^{sp-n} [\fstar]_{W^{s,p}(B(y_0,R/2))}  < C_{s,p,n} \eps.
\]
That is, from \eqref{eq:osc:1} and the definition of $X$ there must be some $R = R(x_0,\eps) > 0$ such that
\begin{equation}\label{eq:PreContinuity}
 \sup_{r \in (0,R/2) \setminus N_{x_0}} \osc_{\partial B(y_0,r)} \fstar < \eps \quad \forall y_0 \in B(x_0,R/2)\,.
\end{equation}
This implies 
\begin{equation}\label{eq:Continuity}
 \osc_{B(x_0,R/4)} \fstar < 2 \eps\,,
\end{equation}
which is continuity. To see \eqref{eq:Continuity}, without loss of generality let $x_0 = 0$. Let $x$ and $y$ be \textit{any} two points in $B(0,R/4) \setminus \{ 0 \}$. Then there exist $r \in (0,|x|)$ and $t \in (0,|y|)$ such that 
\begin{equation*}
\osc_{\p B(r \frac{x}{|x|},|x|-r)} \fstar < \veps\,, \quad \osc_{\p B(t \frac{y}{|y|},|y|-t)} \fstar < \veps\,, \quad \text{ and } \p B(r \frac{x}{|x|},|x|-r) \cap \p B(t \frac{y}{|y|},|y|-t) \neq \emptyset\,.
\end{equation*}
If this is not the case, then by \eqref{eq:PreContinuity} and since the maps $r \mapsto \p B(r x/|x|,|x|-r)$ and $t \mapsto \p B(t y/|y|,|y|-t)$  are continuous it follows that some open interval must reside within the set $N_{x_0}$, a contradiction. Now let $z \in B(0,R/4)$ be a point in the intersection; we have
\begin{equation*}
\begin{split}
|\fstar(x) - \fstar(y)| &\leq |\fstar(x) - \fstar(z)| + |\fstar(z) - \fstar(y)|  \\
	&\leq \osc_{\p B(r \frac{x}{|x|},|x|-r)} \fstar + \osc_{\p B(t \frac{y}{|y|},|y|-t)} \fstar < 2 \veps\,.
\end{split}
\end{equation*}
This holds for \textit{any} $x$ and $y$ not equal to zero. If one of the two points is the center of $B(0,R/4)$ (without loss of generality $y = 0$) then repeat the argument with the set $\p B(t \frac{x}{|x|},t)$ for $t \in (0,|x|)$ in place of 
$\p B(t \frac{y}{|y|},|y|-t)$. Thus \eqref{eq:Continuity} is proved.

\end{proof}

\section{Degree and Monotonicity estimates}\label{s:degreemon}

Let $B=B(x_0,r) \subset \R^n$ and let $f: \partial B \to \R^n$ be continuous. For $p \not \in f(\partial B)$ define the \textit{degree}
\[
 \deg(f,\partial B,p) := \deg_{\S^{n-1}}(\psi)
\]
where 
\[
 \psi := \frac{f\brac{\frac{x-x_0}{r}}-p}{\abs{f\brac{\frac{x-x_0}{r}}-p}} : \S^{n-1} \to \S^{n-1}
\]
and $\deg_{\S^{n-1}}$ computes the homotopy group of $\psi$ in $\pi_{n-1}(\S^{n-1}) = \Z$.

The main topological ingredient is the following lemma (which is well-known). Items (i) and (iii) are essentially a rewritten version of \cite[Lemma 5.1]{BHM19}, and (ii) is a consequence of (i) motivated by \cite{S88,GHP17,LS19}. 
\begin{lemma}\label{lma:TopMonotonicity}
Let $\Omega \subset \R^n$ be an open set. 

Assume that $B_1 := B(x_1,r_1)$ and $B_2 := B(x_2,r_2) \subset\subset \Omega$ are two open balls and $f,f_k : \partial B_1 \cup \partial B_2 \to \R^n$ be continuous maps, $k \in \N$ such that $f_k$ uniformly converges to $f$ on $\partial B_1 \cup \partial B_2$.

If for any $k \in \N$, the map $f_k$ can be extended to a homeomorphism $F_k: \Omega \to \R^n$ then the following hold:

\begin{enumerate}[(i)]
	
 \item\label{item:DegPf:1} If $B_1 \subset B_2$ then \[
 		\begin{split}
        &f(\partial B_1) \cup \left \{p \in \bbR^n \setminus f(\p B_1) \, : \, \deg(f,\partial B_1,p) \neq 0 \right \} \\
        &\qquad \subset 
        f(\partial B_2) \cup \left \{p \in \bbR^n \setminus f(\p B_2) \, : \, \deg(f,\partial B_2,p) \neq 0 \right \}
        \end{split}
       \]
\item\label{item:DegPf:2} If $B_1 \subset B_2$ then we have monotonicity of oscillation,
\[
                                \osc_{\partial B_1} f \leq 8\, \osc_{\partial B_2} f
\]
and
\[
 \diam \left \{p \in \bbR^n \setminus f(\p B_1) \, : \, \deg(f,\partial B_1,p) \neq 0 \right \} \leq 8\, \osc_{\partial B_2} f.
\]
\item\label{item:DegPf:3} If $B_1 \cap B_2 = \emptyset$ then \[
        \left \{p \in \R^n \backslash f(\partial B_1):\, \deg(f,\partial B_1,p) \neq 0 \right \} \cap 
        \left \{p \in \R^n \backslash f(\partial B_2):\, \deg(f,\partial B_2,p) \neq 0 \right \} = \emptyset
       \]
\end{enumerate}
\end{lemma}

\begin{proof}

To prove \ref{item:DegPf:1}, assume $B_1 \subset B_2$ and let \[
    p \in f(\partial B_1) \cup \left \{p \in \bbR^n \setminus f(\p B_1) :\, \deg(f,\partial B_1,p) \neq 0 \right \}.
 \]
If $p \in f(\partial B_2)$ there is nothing to show, so we may assume that $p \not \in f(\partial B_2)$. By uniform convergence $p \not \in f_k(\partial B_2)$ for all large $k$.
 
We use a contradiction argument; assume that $\deg(f,\partial B_2,p) = 0$. By the uniform convergence and since $p \not \in f_k(\partial B_2)$ we have that $\deg(f_k,\partial B_2,p) = 0$ for large $k$.

Let $F_k: \Omega \to \R^n$ be a homeomorphism such that $f_k = F_k \Big |_{\partial B_2}$. Then $\deg(f_k,\partial B_2,p) = 0$ implies that $p \not \in F_k(\overline{B_2})$.
Since $\overline{B_1} \subset B_2$ this implies that $p \not \in F_k(\overline{B_1})$ and thus 
 \[
  \deg(f_k,\partial B_1,p) = 0 \quad \text{ for large } k\,.
 \]
This leads to a contradiction as $k \to \infty$ unless $p \in f(\partial B_1)$.
However since $F_k: \overline{B_2} \to \R^n$ is a homeomorphism, it is an open map so if $p \in (\partial B_1) \backslash F_k(\overline{B_2})$ there must be $q_k \in \partial B_2$ such that 
\[
 \dist(p,F_k(\overline{B_2})) = |p-f_k(q_k)|.
\]
Since $p \not \in f(\partial B_2)$, we conclude via uniform convergence that
\[
 \liminf_{k \to \infty} \dist(p,F_k(\overline{B_2})) >0
 \]
and thus
\[
 \dist(p,f(\partial B_1)) = \liminf_{k \to \infty} \dist(p,f_k(\partial B_1)) \geq \liminf_{k \to \infty} \dist(p,F_k(\overline{B_2})) > 0,
\]
consequently $p \not \in f(\partial B_1)$.
 
%
%
To prove \ref{item:DegPf:2}, we have that 
\[
 f(\partial B_1) \subset f(\partial B_2) \cup  \left \{p \in \bbR^n \setminus f(\p B_2) \,:\, \deg(f,\partial B_2,p) \neq 0 \right \}.
\]
Let $D := \diam(f(\partial B_2))$ and pick any $x_0 \in \partial B_2$. then $f(\partial B_2) \subset B(f(x_0),3D)$. Moreover, let $\pi: \R^n \to B(f(x_0),4D)$ be Lipschitz such that $\pi\Big |_{B(f(x_0),3D)} = id$. Since the degree depends only on the boundary values, for any $p \not \in f(\partial _2)$,
\[
 \deg(f,\partial B_2,p) = \deg(\pi \circ f,\partial B_2,p).
\]
Since a necessary condition for the degree to be nonzero in a point $p$ is that $p$ belongs to the image, we conclude that 
\[
 \left \{p \in \bbR^n \setminus f(\p B_2) \,:\, \deg(f,\partial B_2,p) \neq 0 \right \} \subset B(f(x_0),4D).
\]
In conclusion, we have shown
\[
 f(\partial B_1) \subset B(f(x_0),4D)
\]
and thus
\[
 \diam( f(\partial B_1) ) \leq 8D = 8 \diam(f(\partial B_2)).
\]

For \ref{item:DegPf:3}, assume that $p \in \R^n \backslash \brac{f(\partial B_1) \cup f(\partial B_2)}$ and
\[
 \deg(f,\partial B_1,p) \neq 0, \quad \deg(f,\partial B_2,p) \neq 0.
\]
By uniform convegence, $p \in \R^n \backslash \brac{f_k(\partial B_1) \cup f_k(\partial B_2)}$ for eventually all $k \in \N$, and 
\[
 \deg(f_k,\partial B_1,p) \neq 0, \quad \deg(f_k,\partial B_2,p) \neq 0.
\]
This means that $p \in F_k(B_1) \cap F_k(B_2)$ which is a contradiction to $F_k$ being a homeomorphism.

\end{proof}

\section{Corollaries for Limits of Homeomorphisms}\label{s:coforlemma31}

We need the following result, which is a fractional analogue of \cite[Lemma 2.9]{MS95}:

\begin{lemma}\label{lma:ConvergenceOnSpheres}
Let $n \geq 2$, and let $p \in (1,\infty)$ and $s \in (0,1)$. Suppose that $\Omega \subset \bbR^n$ is a bounded domain, and let 
\begin{equation}
f_k \rightharpoonup f \text{ in } W^{s,p}(\Omega;\bbR^n)\,.
\end{equation}
Let $x_0 \in \Omega$, and define $r_{x_0} := \dist(x_0,\p \Omega)$. Then there exists a set $N_{x_0} \subset \bbR$ with $\scL^1(N_{x_0}) = 0$ such that for any $r \in (0,r_{x_0})\setminus N_{x_0}$ there exists a subsequence $f_k$ such that
\begin{equation}\label{eq:ConvergenceOnSpheres}
\fstar_k \rightharpoonup \fstar \text{ in } W^{s,p}(\p B(x_0,r);\bbR^n)\,.
\end{equation}
If $sp > n-1$ then
\begin{equation}\label{eq:UnifConvOnSpheres}
\fstar_k \rightrightarrows \fstar \text{ on } \p B(x_0,r)\,.
\end{equation}
In general the subsequence depends on $r$.
\end{lemma}

\begin{proof}
First, by compact embedding there is a subsequence $f_k \to f$ in $L^p(B(x_0,r_{x_0});\bbR^n)$ and so Fubini's theorem implies
\begin{equation}
\fstar_k \to \fstar \text{ in } L^p(\p B(x_0,r);\bbR^n)\,, \quad \text{ for every } r \in (0,r_{x_0}) \setminus N_1 \text{ with } \scL^1(N_1) = 0\,.
\end{equation}
Next, define
\begin{equation}
\Phi_k(r) := \iintdm{\p B(x_0,r)}{\p B(x_0,r)}{\frac{|\fstar_k(x)-\fstar_k(y)|^p}{|x-y|^{n-1+sp}}}{\scH^{n-1}(y)}{\scH^{n-1}(x)}\,,
\end{equation}
with
\begin{equation}
\Phi(r) := \liminf_{k \to \infty} \Phi_k(r)\,.
\end{equation}
Then by Fatou's Lemma and by \Cref{lma:TangentDerivEstimate}
\begin{equation*}
\int_{r/2}^{r} \Phi(r) \, \rmd r \leq \liminf_{k \to \infty} \int_{r/2}^{r} \Phi_k(r) \, \rmd r \leq \liminf_{j \to \infty} [\fstar_k]_{W^{s,p}(B(x_0,r))}^p < \infty
\end{equation*}
for every $r \in (0,r_{x_0})$. Define $N_2 := \Vbrack{r \in (0,r_{x_0}) \, : \, \Phi(r) = \infty }$, and define $N_{x_0} := N_1 \cup N_2$; note $\scL^1(N_{x_0}) = 0$. Then let $r \in (0,r_{x_0}) \setminus N_{x_0}$, and choose a subsequence (not relabeled) satisfying
\begin{equation*}
\Phi(r) = \lim_{k \to \infty} \Phi_k(r)\,.
\end{equation*}
Then $\fstar_k \to \fstar$ strongly in $L^p(\p B(x_0,r);\bbR^n)$ and $\lim_{k \to \infty} [\fstar_k]_{W^{s,p}(\p B(x_0,r))} < \infty$, and so \eqref{eq:ConvergenceOnSpheres} follows.


In the event that $sp > n-1$ the uniform convergence follows from the compact Sobolev embedding theorem.
\end{proof}

The following is a corollary of the Sobolev compact embedding theorem, \Cref{lma:TopMonotonicity} and Proposition \ref{pr:osciscontinuity}:

\begin{corollary}\label{cor:TopImage}
Let $f_k \in W^{s,p}(\Omega;\R^n)$ be a sequence of homeomorphisms weakly converging in $W^{s,p}(\Omega;\R^n)$ to $f$. If $sp > n-1$ there exists a set $\Sigma \subset \Omega$ with $\scH^{n-sp}(\Sigma) = 0$ such that
\begin{enumerate}[(i)]
 \item $f^\ast$ is continuous in $\Omega \backslash \Sigma$, and
 \item The set$\{ \fstar(x) \}$ coincides with the topological image $(\fstar)^T(x)$ for every $x \in \Omega \backslash \Sigma$, where $(\fstar)^T(x)$ is defined as
\[
(\fstar)^T(x) := \bigcap_{r \in (0,r_x) \backslash N_x} f^{\ast}(\partial B(x,r)) \cup  \{p \in \bbR^n \setminus \fstar(\p B(x,r)) \, : \, \deg(f^{\ast},B(x,r),p) \neq 0\}\,,
\]
and $r_x$ and $N_x$ have been defined in \Cref{lma:ConvergenceOnSpheres}.
\end{enumerate}
\end{corollary}

\begin{proof}
By \Cref{lma:ConvergenceOnSpheres} and Corollary \ref{cor:PrecRepOnSphere}, the assumptions of \Cref{lma:TopMonotonicity} are satisfied for every $x_1$ and $x_2 \in \Omega$ and for almost every $r_1 \in (0,r_{x_1}) \setminus N_{x_1}$ and $r_2 \in (0,r_{x_2}) \setminus N_{x_2}$.
It follows that the assumptions of Proposition \ref{pr:osciscontinuity} are satisfied, and so $\fstar$ is continuous on a $\scH^{n-sp}$-null set $\Sigma$, where $\Sigma = A_{sp} \cup X$; see \eqref{eq:LebPtSet} and \eqref{eq:FrostLemmaSet} for the sets' definitions. Thus (i) is proven.

To prove (ii) it suffices to show that 
\begin{enumerate}[(a)]
\item \label{item:TopImage1}
$\fstar(x) \in (\fstar)^T(x)$ for every $x \in \Omega \setminus \Sigma$, and
\item\label{item:TopImage2}
the diameter of the set $(\fstar)^T(x)$
is zero for every $x \in \Omega \setminus \Sigma$.
\end{enumerate}

To see \ref{item:TopImage1} we start by proving the following stronger statement:
\begin{equation}\label{eq:INV}
\begin{split}
&\text{For every } x_0 \in \Omega \text{ and } r \in (0,r_{x_0}) \setminus N_{x_0}\,, \\
\fstar(x) &\in f^{\ast}(\partial B(x_0,r)) \cup  \{p \in \bbR^n \setminus \fstar(\p B(x_0,r)) \, : \, \deg(f^{\ast},B(x_0,r),p) \neq 0\} \\
&\qquad \text{ for every } x \in B(x_0,r) \setminus \Sigma\,. 
\end{split}\tag{a'}
\end{equation}
Then \ref{item:TopImage1} follows easily from \eqref{eq:INV} by choosing $x_0 \in \Omega \setminus \Sigma$. By definition of $\Sigma$ and by \Cref{lma:WeaktoStrongINV} it in turn suffices to show that
\begin{equation}\label{eq:WeakINV}
\begin{split}
&\text{For every } x_0 \in \Omega \text{ and } r \in (0,r_{x_0}) \setminus N_{x_0}\,, \\
\fstar(x) &\in f^{\ast}(\partial B(x_0,r)) \cup  \{p \in \bbR^n \setminus \fstar(\p B(x_0,r)) \, : \, \deg(f^{\ast},B(x_0,r),p) \neq 0\} \\
&\qquad \text{ for every } x \in B(x_0,r) \setminus M \text{ with } \scL^n(M) = 0\,. 
\end{split}\tag{a''}
\end{equation}
Let $\delta > 0$ be arbitrary. Then by the Sobolev compact embedding theorem and by Egorov's theorem there exists a subsequence (not relabeled) $f_k$ converging uniformly to $\fstar$ on $B(x_0,r) \setminus M_{\delta}$ with $\scL^n(M_{\delta}) < \delta$.
Now let $x \in \Omega \setminus M_{\delta}$. It suffices to show that if $\fstar(x) \notin \fstar(\p B(x_0,r))$ then $\deg(\fstar, B(x_0,r), \fstar(x)) \neq 0$. Since $f_k \rightrightarrows \fstar$ on $\p B(x_0,r)$, $\fstar(x) \notin f_k(\p B(x,r))$ for all $k$ sufficiently large. 
So there exists $\veps > 0$ such that $B(\fstar(x),\veps)$ does not intersect $\fstar(\p B(x_0,r))$ or $f_k(\p B(x_0,r))$ for $k$ sufficiently large. Then since the $f_k$ are homeomorphisms, it must be that $\deg(f_k, \p B(x_0,r), p)$ is a nonzero constant for all $k$ sufficiently large and for all $p \in B(\fstar(x),\veps)$.
In addition, $f_k \rightrightarrows \fstar$ on $B(x_0,r) \setminus M_{\delta}$ so $f_k(x) \in B(\fstar(x),\veps)$ for $k$ sufficiently large, uniformly in $x$. Thus the continuity of the degree yields
\begin{equation*}
\deg(\fstar, B(x_0,r), \fstar(x)) = \lim_{k \to \infty} \deg(f_k, B(x_0,r), f_k(x))\,.
\end{equation*}
Since $\deg(f_k, B(x_0,r), f_k(x))$ is a nonzero constant for all $k$ sufficiently large, we have proved that 
\begin{equation*}
\begin{split}
&\text{For every } x_0 \in \Omega \text{ and } r \in (0,r_{x_0}) \setminus N_{x_0}\,, \\
\fstar(x) &\in f^{\ast}(\partial B(x_0,r)) \cup  \{p \in \bbR^n \setminus \fstar(\p B(x_0,r)) \, : \, \deg(f^{\ast},B(x_0,r),p) \neq 0\} \\
&\qquad \text{ for every } x \in B(x_0,r) \setminus M_{\delta} \text{ with } \scL^n(M_{\delta}) < \delta\,. 
\end{split}
\end{equation*}
Since $\delta > 0$ is arbitrary \eqref{eq:WeakINV} is proved.

To see \ref{item:TopImage2}, let $x_0 \in \Omega \setminus \Sigma$, and let $\veps > 0$. Then by definition of the set $X$ there exists $R = R(x_0,\veps) \in (0,r_{x_0})$ such that
\begin{equation*}
R^{sp-n} [\fstar]_{W^{s,p}(B(x_0,R))} < \veps\,.
\end{equation*}
So by  \Cref{lma:TopMonotonicity} \ref{item:DegPf:2}  and \eqref{eq:osc:1}
\begin{equation*}
\diam (\fstar)^T(x_0) \leq \diam \Big( \fstar(\p B(x_0,r)) \cup \{ p \, : \, \deg(\fstar, B(x_0,r), p) \neq 0 \} \Big) < C \veps
\end{equation*}
for every $r \in (0,R/4) \setminus N_{x_0}$. Therefore by definition 
$\diam (\fstar)^T(x_0) < \veps$. The proof is complete.

\begin{remark}
We can define a representative $\wh{f}$ of $f$ as
\begin{equation}
\wh{f}(x) :=
	\begin{cases}
		\fstar(x)\,, \qquad& x \in \Omega \setminus \Sigma\,, \\
		\text{any element of } f^T(x)\,, \qquad& \text{ otherwise, }
	\end{cases}
\end{equation}
Then $\wh{f}$ agrees with $\fstar$ everywhere outside $\Sigma$, and $\wh{f}$ has the added property that $\wh{f}(x) \in (\wh{f})^T(x)$ \textit{for every} $x \in \Omega$. 
\end{remark}

\end{proof}

\section{Proof of Theorem~\ref{thm:MainThm}}\label{s:thmainproof}

\begin{proof}[Proof of Theorem \ref{thm:MainThm}]
We proceed identically to \cite{BHM19}. Assume that $f = \wh{f}$. We argue by contradiction; suppose that there is a $\delta > 0$ such that the set
\begin{equation}
\Gamma := \Vbrack{y \in \bbR^n \, : \, \diam(f^{-1}(\vbrack{y})) > 0} 
\end{equation}
satisfies $\scH^{\frac{n-1}{s}+\delta}(\Gamma) > 0$. Then there exists $K \in \bbN$ such that the set
\begin{equation}
\Gamma_K := \Vbrack{y \in \bbR^n \, : \, \diam(f^{-1}(\vbrack{y})) > \frac{1}{K}}
\end{equation}
satisfies $\scH^{\frac{n-1}{s}+\delta}(\Gamma_K) > 0$, since $F = \bigcup_{k \in \bbN} \Gamma_k$.
For each $x$ there exists $r < \frac{1}{2K}$ such that $f|_{\p B(x,r)} \in W^{s,p}(\p B(x,r);\bbR^n) \cap C^0(\p B(x,r);\bbR^n)$ by Corollary \ref{lma:ConvergenceOnSpheres}. Then choosing a covering of $\Omega$ with such a collection  $\cB := \vparen{B(x_i,r_i)}_{i=1}^{\infty}$, by Theorem \ref{thm:DimOfImage} we have $\dim_{\scH}(f(\p B(x_i,r_i))) < \frac{n-1}{s}$, so $\scH^{\frac{n-1}{s}+\delta}(f(\p B(x_i,r_i))) = 0$. Therefore, the set
\begin{equation}
E := \bigcup_{i=1}^{\infty} f(\p B(x_i,r_i))
\end{equation}
satisfies $\scH^{\frac{n-1}{s}+\delta}(E) = 0$. We will show that $\Gamma_K \subset E$, which contradicts the statement $\scH^{\frac{n-1}{s}+\delta}(\Gamma_K) > 0$.

Assume $y \in \Gamma_K \setminus E$. Then there must exist $z_1$ and $z_2$ in $\Omega$ with $f(z_1) = f(z_2) = y$, with $\dist(z_1, z_2) > \frac{1}{K}$. Fix an element $B(x_i, r_i)$ from the collection $\cB$ with $z_1 \in B(x_i, r_i)$ and $z_2 \notin B(x_i,r_i)$.
Combining \Cref{lma:TopMonotonicity}\ref{item:DegPf:1} with the fact that 
\begin{equation*}
f(x) \in f^T(x) \subset f(\p B(x,r)) \cup \{ p \in \bbR^n \setminus f(\p B(x,r)) \, : \, \deg(f,\p B(x,r),p) \neq 0 \}
\end{equation*}
for all $x \in \Omega$ and for $r \in (0,\dist(x,\p \Omega)) \setminus N_x$, we get
\begin{equation*}
y = f(z_1) \in f(\p B(x_i,r_i)) \cup \{ p \in \bbR^n \setminus f(\p B(x_i,r_{x_i})) \, : \, \deg(f, B(x_i,r_{x_i}),p) \neq 0 \}\,.
\end{equation*}
However $y \notin E$ so $y \notin f(\p B(x_i,r_i))$, and thus
\begin{equation*}
y = f(z_1) \in \{ p \in \bbR^n \setminus f(\p B(x_i,r_{x_i})) \, : \, \deg(f, B(x_i,r_{x_i}),p) \neq 0 \}\,.
\end{equation*}
At the same time, a similar argument using \Cref{lma:TopMonotonicity}\ref{item:DegPf:3} gives
\begin{equation*}
y = f(z_2) \in f^T(z_2) \subset \bbR^n \setminus \{ p \in \bbR^n \setminus f(\p B(x_i,r_{x_i})) \, : \, \deg(f, B(x_i,r_{x_i}),p) \neq 0 \}\,,
\end{equation*}
which is a contradiction.
\end{proof}

\appendix

\section{Proof of Lemma~\ref{lma:TangentDerivEstimate}}\label{s:lemma2}

\begin{proof}
It suffices to prove \eqref{eq:TangentDerivEstimate} for $a = 0$ and $r = 1$. In the case of general $a$ and $r$ we can apply \eqref{eq:TangentDerivEstimate} for $a=0$, $r=1$ to the function 
\begin{equation*}
g(x) := f(a+rx) \in W^{s,p}(B(0,1))
\end{equation*}
and obtain \eqref{eq:TangentDerivEstimate} for general $a$ and $r$ by change of variables. 

Since the function $f - (f)_B$ also belongs to $W^{s,p}(B(0,1))$ we can assume without loss of generality that
\begin{equation*}
\fint_{B(0,1)} f \, \rmd x = 0\,.
\end{equation*}
Thus by the Poincar\'e inequality it suffices to show that there exists a constant $C = C(n,s,p) > 0$ such that
\begin{equation}\label{eq:TangentDerivEstimate:Reduced}
\begin{split}
\int_{1/2}^1 \iintdm{\bbS^{n-1}}{\bbS^{n-1}}{\rho^{n-1-sp}\frac{|f(\rho x)-f(\rho y)|^p}{|x-y|^{n-1+sp}}}{\scH^{n-1}(y)}{\scH^{n-1}(x)} \, \rmd \rho \leq C \Vnorm{f}_{W^{s,p}(B(0,1))}^p\,;
\end{split}
\end{equation}
note that we used polar coordinates to rewrite the integral.

We prove \eqref{eq:TangentDerivEstimate:Reduced} by splitting the domain of the left-hand side integral and estimating each piece. Each domain of integration is locally homeomorphic to a Euclidean ball in $\bbR^{n-1}$, which allows us to apply translation arguments in the spirit of \cite[Lemma 7.44]{A75}. Any local diffeomorphism between $\bbS^{n-1}$ and $\bbR^{n-1}$ will do, but we make this argument explicit by using stereographic projection.

\underline{Step 1:} To this end, define for each $\mu \in [0,1)$ the spherical cap $H_{\mu} := \{ x \in \bbS^{n-1} \, : \, x_n < \mu \}$.
We will show that for every $\mu \in [0,1)$ there exists a constant $C = C(n,s,p)$ such that
\begin{equation}\label{eq:SphereEstimate:Step1}
\begin{split}
\int_{1/2}^1 &\iintdm{H_{\mu}}{H_{\mu}}{\rho^{n-1-sp} \frac{|f(\rho x)-f(\rho y)|^p}{|x-y|^{n-1+sp}}}{\scH^{n-1}(y)}{\scH^{n-1}(x)} \, \rmd \rho \leq C \Vparen{\frac{1+\mu}{1-\mu}}^{1+sp} \Vnorm{f}_{W^{s,p}(B(0,1))}^p\,.
\end{split}
\end{equation}

Throughout the proof we write $B_{n-1}(0,\lambda)$ for any $\lambda > 0$ as the ball in $\bbR^{n-1}$ centered at $0$ of radius $\lambda$. We next establish notation for the stereographic projection $\psi : \bbR^{n-1} \to \bbS^{n-1} \setminus \{(0,\ldots, 0, 1)\}$ to prove \eqref{eq:SphereEstimate:Step1}. Details on the stereographic projection can be found in several places, for instance \cite[Appendix D.6]{G08}. We use the definition 
\begin{equation*}
\psi(x_1,\ldots,x_{n-1}) := \Vparen{\frac{2x_1}{1+|x|^2}, \ldots, \frac{2x_{n-1}}{1+|x|^2}, \frac{|x|^2-1}{1+|x|^2} }
\end{equation*}
so that we have the correspondence of domains
\begin{equation*}
\psi(B_{n-1}(0,\lambda)) = H_{\mu}\,,
\quad \text{ where } \quad
\lambda = \sqrt{\frac{1+\mu}{1-\mu}} \,.
\end{equation*}
The formula for the Jacobian $J_{\psi}(x) := \Vparen{ \frac{2}{1+|x|^2} }^{n-1}$ will be used throughout in order to ensure that quantities such as $J_{\psi}(x-y)$ and $|J_{\psi}(x)-J_{\psi}(y)|$ remain bounded above and below uniformly for $x$ and $y$ in $B_{n-1}(0,\lambda)$, with bounds depending only on $n$ and $\lambda$.

To prove \eqref{eq:SphereEstimate:Step1} we need to show that for every $\lambda \in [1,\infty)$ and for every ball $B_{n-1}(0,\lambda) \subset \bbR^{n-1}$
\begin{equation}\label{eq:SphereEstimate:Step1:SterProj}
\begin{split}
\int_{1/2}^1 &\iintdm{B_{n-1}(0,\lambda)}{B_{n-1}(0,\lambda)}{\rho^{n-1-sp} \frac{|f(\rho \psi(x))-f(\rho \psi(y))|^p}{|\psi(x)-\psi(y)|^{n-1+sp}} \\
	&\qquad \times J_{\psi}(x) J_{\psi}(y)}{y}{x} \, \rmd \rho \precsim_{n,s,p} \lambda^{2+2sp} \Vnorm{f}_{W^{s,p}(B(0,1))}^p\,.
\end{split}
\end{equation}
We proceed using a technique found in \cite[Lemma 7.44]{A75}. Let $\sigma \in [1/2,1]$, and integrate the inequality 
\begin{equation*}
\begin{split}
|f(\rho \psi(x))-f(\rho \psi(y))|^p &\precsim_p |f(\rho \psi(x)) - f(\sigma\psi(\textstyle{\frac{x+y}{2}}))|^p \\
	&\qquad + |f(\sigma\psi(\textstyle{\frac{x+y}{2}}))-f(\rho \psi(y))|^p
\end{split}
\end{equation*}
with respect to $\sigma$ over the ball $B(r,\textstyle{\frac{|\psi(x)-\psi(y)|}{2}}) \cap [\textstyle{\frac{1}{2}},1] \subset \bbR$ to get
\begin{equation*}
\begin{split}
& |f(\rho \psi(x))-f(\rho \psi(y))|^p  \\
	&\quad \precsim_p \frac{1}{|\psi(x)-\psi(y)|} \int_{ \{ |\sigma-\rho|\leq \frac{|\psi(x)-\psi(y)|}{2} \} \cap [1/2,1] } |f(\rho\psi(x)) - f(\sigma\psi(\textstyle{\frac{x+y}{2}}))|^p \, \rmd \sigma \\
	&\qquad + \frac{1}{|\psi(x)-\psi(y)|} \int_{ \{ |\sigma-\rho|\leq \frac{|\psi(x)-\psi(y)|}{2} \} \cap [1/2,1] } |f(\sigma\psi(\textstyle{\frac{x+y}{2}}))-f(\rho \psi(y))|^p \, \rmd \sigma\,.
\end{split}
\end{equation*}
Set for $\eta \in \bbR^{n-1}$
\begin{equation*}
\Upsilon(\eta) := \int_{ \{ |\sigma-\rho|\leq \frac{|\psi(x)-\psi(y)|}{2} \} \cap [1/2,1] } |f(\rho \psi(\eta)) - f(\sigma\psi(\textstyle{\frac{x+y}{2}}))|^p \, \rmd \sigma\,;
\end{equation*}
therefore
\begin{equation}\label{eq:SphereEstimateProof:0.5}
\begin{split}
&\int_{1/2}^1 \rho^{n-1-sp} \iintdm{B_{n-1}(0,\lambda)}{B_{n-1}(0,\lambda)}{\frac{|f(\rho\psi(x))-f(\rho\psi(y))|^p}{|\psi(x)-\psi(y)|^{n-1+sp}} J_{\psi}(x) J_{\psi}(y)}{y}{x} \, \rmd \rho \\
	&\quad \precsim_p \int_{1/2}^1 \rho^{n-1-sp} \iintdm{B_{n-1}(0,\lambda)}{B_{n-1}(0,\lambda)}{\frac{\Upsilon(x)+\Upsilon(y)}{|\psi(x)-\psi(y)|^{n+sp}} J_{\psi}(x) J_{\psi}(y)}{y}{x} \, \rmd \rho := \rmI +\mathrm{II}\,.
\end{split}
\end{equation}
Now by change of variables and using the formula for $J_{\psi}$ as well as the formula
\begin{equation*}
|\psi(x) -\psi(y)| = \frac{2|x-y|}{(1+|x|^2)^{1/2} (1+|y|^2)^{1/2}}
\end{equation*}
which is valid for all $x$, $y \in \bbR^{n-1}$, we have
\begin{equation*}
\begin{split}
\rmI &= \int_{1/2}^1 \int_{B_{n-1}(0,\lambda)} \int_{B_{n-1}(0,\lambda)} \int_{ \{ |\sigma-\rho|\leq \frac{|\psi(x)-\psi(y)|}{2} \} \cap [1/2,1] } \rho^{n-1-sp} \frac{|f(\rho\psi(x))-f(\sigma\psi({\textstyle \frac{x+y}{2}}))|^p}{|\psi(x)-\psi(y)|^{n+sp}} \\
	&\qquad \hspace{0.3\textwidth} \times J_{\psi}(x) J_{\psi}(y) \, \rmd \sigma \, \rmd y \, \rmd x \, \rmd \rho \\
	& = \int_{1/2}^1 \int_{1/2}^1 \int_{B_{n-1}(0,\lambda)} \int_{\Vbrack{|2z-x| \leq \lambda} \cap \Vbrack{|\psi(x)-\psi(2z-x)| \geq 2|\sigma-\rho|}} \rho^{n-1-sp} \frac{|f(\rho \psi(x))-f(\sigma\psi(z))|^p}{|\psi(x)-\psi(2z-x)|^{n+sp}} \\
	&\qquad \hspace{0.3\textwidth} \times J_{\psi}(x) J_{\psi}(2z-x) \, \rmd z \, \rmd x \, \rmd \sigma  \, \rmd \rho \\
	& = \frac{1}{2^{(n+sp)/2}} \int_{1/2}^1 \int_{1/2}^1 \int_{B_{n-1}(0,\lambda)} \int_{\Vbrack{|2z-x| \leq \lambda} \cap \Vbrack{|\psi(x)-\psi(z)| \geq G(x,z) |\sigma-\rho|}} \rho^{n-1-sp} \frac{|f(\rho \psi(x))-f(\sigma\psi(z))|^p}{|\psi(x)-\psi(z)|^{n+sp}} \\
	&\qquad \hspace{0.3\textwidth} \times J_{\psi}(x) J_{\psi}(z) G(x,z)^{2+sp-n} \, \rmd z \, \rmd x \, \rmd \sigma  \, \rmd \rho\,,
\end{split}
\end{equation*}
where
\begin{equation*}
G(x,z) :=  \Vparen{\frac{1+|2z-x|^2}{1+|x|^2}}^{1/2}\,.
\end{equation*}
Now, since $|x| \leq \lambda$ and $|2z-x| \leq \lambda$
the uniform bound 
$
\frac{1}{\sqrt{1+\lambda^2}} \leq G(x,z) \leq \sqrt{1+ \lambda^2}
$
holds, 
and so $\rmI$ can be majorized by
\begin{equation}\label{eq:SphereEstimateProof:3}
\begin{split}
	& C(n,s,p) \lambda^{2+sp-n} \int_{1/2}^1 \int_{1/2}^1 \int_{B_{n-1}(0,\lambda)} \int_{\Vbrack{|2z-x| \leq \lambda} \cap \Vbrack{|\psi(x)-\psi(z)| \geq \frac{|\sigma-\rho|}{\sqrt{1+\lambda^2}}} } \rho^{n-1-sp} \\
		&\qquad \hspace{0.2\textwidth} \times\frac{|f(\rho \psi(x))-f(\sigma\psi(z))|^p}{|\psi(x)-\psi(z)|^{n+sp}}  J_{\psi}(x) J_{\psi}(z)  \, \rmd z \, \rmd x \, \rmd \sigma \, \rmd \rho\,.
\end{split}
\end{equation}
Finally, on the domain of integration in \eqref{eq:SphereEstimateProof:3} the estimate
\begin{equation*}
\begin{split}
|\rho \psi(x)-\sigma \psi(z)| &\leq \rho|\psi(x)-\psi(z)| + |\psi(z)| |\rho-\sigma| \\
	&\leq |\psi(x)-\psi(z)| + (1+\lambda^2)^{1/2} |\psi(x)-\psi(z)| \leq 2 \lambda |\psi(x)-\psi(z)|
\end{split}
\end{equation*}
holds, and since $\rho$ and $\sigma$ are bounded away from zero
\begin{equation*}
\begin{split}
\rmI & \leq C \lambda^{2+2sp} \int_{1/2}^1 \int_{1/2}^1 \int_{B_{n-1}(0,\lambda)} \int_{\Vbrack{|2z-x| \leq \lambda} \cap \Vbrack{|\psi(x)-\psi(z)| \geq \frac{|\sigma-\rho|}{\sqrt{1+\lambda^2}}} } \rho^{n-1} \sigma^{n-1}  \\
		&\qquad \hspace{0.2\textwidth} \times  \frac{|f(\rho \psi(x))-f(\sigma\psi(z))|^p}{|\rho \psi(x)-\sigma\psi(z)|^{n+sp}} J_{\psi}(x) J_{\psi}(z)  \, \rmd z \, \rmd x \, \rmd \sigma \, \rmd \rho \\
	& \leq C \lambda^{2+2sp} \int_0^1 \int_0^1 \int_{\bbR^{n-1}} \int_{\bbR^{n-1}} \rho^{n-1} \sigma^{n-1}  \frac{|f(\rho \psi(x))-f(\sigma\psi(z))|^p}{|\rho \psi(x)-\sigma\psi(z)|^{n+sp}} \\
		&\qquad \hspace{0.2\textwidth} \times J_{\psi}(x) J_{\psi}(z)  \, \rmd z \, \rmd x \, \rmd \sigma \, \rmd \rho \\
	& = C(n,s,p) \lambda^{2+2sp} \iintdm{B(0,1)}{B(0,1)}{\frac{|f(x)-f(y)|^p}{|x-y|^{n+sp}}}{y}{x}\,.
\end{split}
\end{equation*}
A similar estimate holds for the quantity $\mathrm{II}$ in \eqref{eq:SphereEstimateProof:0.5}.
Therefore \eqref{eq:SphereEstimate:Step1:SterProj}, and thus \eqref{eq:SphereEstimate:Step1}, is proved.

\underline{Step 2:} 
We conclude the proof. Split the integral on the left-hand side of \eqref{eq:TangentDerivEstimate:Reduced} via change of variables and symmetry as
\begin{equation}\label{eq:SphereEstimate:Pieces}
\begin{split}
\int_{1/2}^1 & \iintdm{\bbS^{n-1}}{\bbS^{n-1}}{\rho^{n-1-sp} \frac{|f(\rho x)-f(\rho y)|^p}{|x-y|^{n-1+sp}}}{\scH^{n-1}(y)}{\scH^{n-1}(x)} \, \rmd \rho \\
	&= \int_{1/2}^1 \int_{H_0} \int_{H_0} \cdots +  \int_{1/2}^1 \int_{\bbS^{n-1} \setminus H_0} \int_{\bbS^{n-1} \setminus H_0} \cdots + 2 \int_{1/2}^1 \int_{\bbS^{n-1} \setminus H_0} \int_{H_0} \cdots \\
	&:= \rmI + \mathrm{II} + \mathrm{III}\,.
\end{split}
\end{equation}
Clearly by \eqref{eq:SphereEstimate:Step1} with $\mu = 0$
\begin{equation}\label{eq:SphereEstimate:Proof:EstimateOfI}
\rmI \precsim_{n,s,p} [f]_{W^{s,p}(B(0,1))}^p\,.
\end{equation}
Now, let $Q : \bbR^{n-1} \to \bbR^{n-1}$ be the matrix $\diag(1,1,\ldots,1,-1)$.
Setting $h(x) = f(Qx)$ for any $x \in B_1(0)$, a change of variables gives
\begin{equation*}
\mathrm{II} = \int_{1/2}^1 \iintdm{H_0}{H_0}{\rho^{n-1-sp} \frac{|h(\rho x)-h(\rho y)|^p}{|x-y|^{n-1+sp}}}{\scH^{n-1}(y)}{\scH^{n-1}(x)} \, \rmd \rho\,.
\end{equation*}
Thus by \eqref{eq:SphereEstimate:Step1} with $\mu = 0$ and by another change of variables
\begin{equation}\label{eq:SphereEstimate:Proof:EstimateOfII}
\mathrm{II} \leq C \iintdm{B(0,1)}{B(0,1)}{\frac{|h(x)-h(y)|^p}{|x-y|^{n+sp}}}{y}{x} = C \iintdm{B(0,1)}{B(0,1)}{\frac{|f(x)-f(y)|^p}{|x-y|^{n+sp}}}{y}{x}\,.
\end{equation}
For the last integral, we have
\begin{equation*}
\begin{split}
\mathrm{III} &= 2 \int_{1/2}^1 \int_{H_{1/2} \setminus H_0 } \int_{H_0} \cdots + 2 \int_{1/2}^1 \int_{\bbS^{n-1} \setminus H_{1/2} } \int_{H_0} \cdots  \\
	&:= \mathrm{III}_1 + \mathrm{III}_2\,.
\end{split}
\end{equation*}
Using that $H_0 \subset H_{1/2}$ along with  \eqref{eq:SphereEstimate:Step1} for $\mu = 1/2$,
\begin{equation}\label{eq:SphereEstimate:Proof:EstimateOfIII1}
\begin{split}
\mathrm{III}_1 &\leq \int_{1/2}^1 \iintdm{H_{1/2}}{H_{1/2}}{\rho^{n-1-sp} \frac{|f(\rho x)-f(\rho y)|^p}{|x-y|^{n-1+sp}} }{\scH^{n-1}(y)}{\scH^{n-1}(x)} \, \rmd \rho \leq [f]_{W^{s,p}(B(0,1))}^p\,.
\end{split}
\end{equation}
Since $\dist(\overline{\bbS^{n-1} \setminus H_{1/2}},\overline{H_0}) = C(n) > 0$, we have that $|x-y| \geq C(n) > 0$ for all $x \in \bbS^{n-1} \setminus H_{1/2}$ and for all $y \in H_0$, and so the integral $\mathrm{III}_2$ can be estimated by
\begin{equation}\label{eq:SphereEstimate:Proof:EstimateOfIII2}
\mathrm{III}_2 \precsim_{n,s,p} \int_{1/2}^1 \int_{\bbS^{n-1}} |f(\rho x)|^p \, \rmd \scH^{n-1}(x) \, \rmd \rho \precsim_{n,s,p} \Vnorm{f}_{L^p(B(0,1))}\,.
\end{equation}
Combining \eqref{eq:SphereEstimate:Pieces} with estimates \eqref{eq:SphereEstimate:Proof:EstimateOfI}, \eqref{eq:SphereEstimate:Proof:EstimateOfII}, \eqref{eq:SphereEstimate:Proof:EstimateOfIII1} and \eqref{eq:SphereEstimate:Proof:EstimateOfIII2} gives \eqref{eq:TangentDerivEstimate:Reduced}.
\end{proof}

\bibliographystyle{abbrv}
\bibliography{bib}

\begin{thebibliography}{10}

\bibitem{A12}
D.~R. Adams and L.~I. Hedberg.
\newblock {\em Function spaces and potential theory}, volume 314.
\newblock Springer Science \& Business Media, 2012.

\bibitem{A75}
R.~A. Adams.
\newblock {\em Sobolev spaces}, volume~65 of {\em Pure and Applied
  Mathematics}.
\newblock Academic, New York-London, 1975.

\bibitem{BHM17}
M.~Barchiesi, D.~Henao, and C.~Mora-Corral.
\newblock Local invertibility in {S}obolev spaces with applications to nematic
  elastomers and magnetoelasticity.
\newblock {\em Arch. Ration. Mech. Anal.}, 224(2):743--816, 2017.

\bibitem{BHM19}
O.~Bouchala, S.~Hencl, and A.~Molchanova.
\newblock Injectivity almost everywhere for weak limits of sobolev
  homeomorphisms, 2019.

\bibitem{BBM05}
J.~Bourgain, H.~Brezis, and P.~Mironescu.
\newblock Lifting, degree, and distributional {J}acobian revisited.
\newblock {\em Communications on Pure and Applied Mathematics: A Journal Issued
  by the Courant Institute of Mathematical Sciences}, 58(4):529--551, 2005.

\bibitem{E15}
L.~C. Evans and R.~F. Gariepy.
\newblock {\em Measure theory and fine properties of functions}.
\newblock CRC press, 2015.

\bibitem{GHP17}
P.~Goldstein, P.~Haj\l{}asz, and M.~R. Pakzad.
\newblock Finite distortion {S}obolev mappings between manifolds are
  continuous.
\newblock {\em Int. Math. Res. Not. IMRN}, (14):4370--4391, 2019.

\bibitem{G08}
L.~Grafakos.
\newblock {\em Classical {F}ourier Analysis, 2nd Edition}, volume 249.
\newblock 2008.

\bibitem{HH15}
S.~Hencl and P.~Honz{\'\i}k.
\newblock Dimension distortion of images of sets under {S}obolev mappings.
\newblock {\em Ann. Acad. Sci. Fenn. Math 40}, (1):427--442, 2015.

\bibitem{LS19}
S.~Li and A.~Schikorra.
\newblock {$W^{s,\frac{n}{s}}$-maps with positive distributional Jacobians}.
\newblock {\em Pot.A. (accepted)}, 2019.

\bibitem{MS95}
S.~M\"{u}ller and S.~J. Spector.
\newblock An existence theory for nonlinear elasticity that allows for
  cavitation.
\newblock {\em Arch. Rational Mech. Anal.}, 131(1):1--66, 1995.

\bibitem{MST96}
S.~M\"{u}ller, S.~J. Spector, and Q.~Tang.
\newblock Invertibility and a topological property of {S}obolev maps.
\newblock {\em SIAM J. Math. Anal.}, 27(4):959--976, 1996.

\bibitem{DNPV12}
E.~D. Nezza, G.~Palatucci, and E.~Valdinoci.
\newblock Hitchhiker's guide to the fractional {S}obolev spaces.
\newblock {\em Bull. Sci. Math.}, 136(5):521--573, 2012.

\bibitem{RS11}
T.~Runst and W.~Sickel.
\newblock {\em Sobolev spaces of fractional order, {N}emytskij operators, and
  nonlinear partial differential equations}, volume~3.
\newblock Walter de Gruyter, 2011.

\bibitem{S18}
Y.~Sawano.
\newblock {\em Theory of {S}esov spaces}, volume~56.
\newblock Springer, 2018.

\bibitem{T07}
L.~Tartar.
\newblock {\em An introduction to {S}obolev spaces and interpolation spaces},
  volume~3.
\newblock Springer Science \& Business Media, 2007.

\bibitem{S88}
V.~\v{S}ver\'{a}k.
\newblock Regularity properties of deformations with finite energy.
\newblock {\em Arch. Rational Mech. Anal.}, 100(2):105--127, 1988.

\bibitem{Z89}
W.~P. Ziemer.
\newblock {\em Weakly differentiable functions}, volume 120 of {\em Graduate
  Texts in Mathematics}.
\newblock Springer-Verlag, New York, 1989.
\newblock Sobolev spaces and functions of bounded variation.

\end{thebibliography}

\end{document}